\documentclass[10pt,reqno]{amsart}
\usepackage{amsmath,amssymb,mathrsfs,graphicx}
\usepackage{ifthen}
\usepackage[margin=1in]{geometry}
\usepackage{caption}
\usepackage{sidecap}
\usepackage{rotating}
\usepackage{colortbl}

\newcommand{\bbr}{\mathbb{R}}

\provideboolean{shownotes} 
\setboolean{shownotes}{true} 
\newcommand{\margnote}[1]{
\ifthenelse{\boolean{shownotes}}%
{\marginpar{\raggedright\tiny\texttt{#1}}}%
{}%
}
\newcommand{\hole}[1]{
\ifthenelse{\boolean{shownotes}}%
{\begin{center} \fbox{ \rule {.25cm}{0cm} \rule[-.1cm]{0cm}{.4cm}
\parbox{.85\textwidth}{\begin{center} \texttt{#1}\end{center}} \rule
{.25cm}{0cm}}\end{center}} {} }

\graphicspath{{pics/}}

\title[On the singular Cucker--Smale model]{One dimensional singular Cucker--Smale model: uniform-in-time mean-field limit and contractivity }

\author[Choi]{Young-Pil Choi}
\address[Young-Pil Choi]{\newline Department of Mathematics\newline
Yonsei University, 50 Yonsei-Ro, Seodaemun-Gu, Seoul 03722, Republic of Korea}
\email{ypchoi@yonsei.ac.kr}

\author[Zhang]{Xiongtao Zhang}
\address[Xiongtao Zhang]{\newline Center for Mathematical Sciences\newline
 Huazhong University of Science and Technology, Luoyu Road 1037, Wuhan 430074, China}
\email{xtzhang@hust.edu.cn}

\numberwithin{equation}{section}

\newtheorem{theorem}{Theorem}[section]
\newtheorem{lemma}{Lemma}[section]
\newtheorem{corollary}{Corollary}[section]

\newtheorem{remark}{Remark}[section]
\newtheorem{definition}{Definition}[section]

\newcommand{\R}{\mathbb R}

\newcommand{\mc}{\mathcal C}

\newcommand{\bq}{\begin{equation}}
\newcommand{\eq}{\end{equation}}

\newcommand{\lt}{\left}
\newcommand{\rt}{\right}

\newcommand{\pa}{\partial}

\newcommand{\W}{\mathcal{W}}
\newcommand{\intr}{\int_\R}
\newcommand{\intrr}{\iint_{\R \times \R}}

\begin{document}
\allowdisplaybreaks


\subjclass[]{}
\keywords{}

\begin{abstract} 
We analyze the one dimensional Cucker--Smale (in short CS) model with a weak singular communication weight $\psi(x) = |x|^{-\beta}$ with $\beta \in (0,1)$. We first establish a global-in-time existence of measure-valued solutions to the kinetic CS equation. For this, we use a proper change of variable to reformulate the particle CS model as a first-order particle system and provide the uniform-in-time stability for that particle system. We then extend this stability estimate for the singular CS particle system. By using that stability estimate, we construct the measure-valued solutions to the kinetic CS equation globally in time. Moreover, as a direct application of the uniform-in-time stability estimate, we show the quantitative uniform-in-time mean-field limit from the particle system to that kinetic CS equation in $p$-Wasserstein distance with $p \in [1,\infty]$. Our result gives the uniqueness of measure-valued solution in the sense of mean-field limits, i.e., the measure-valued solutions, approximated by the empirical measures associated to the particle system, uniquely exist. Similar results for the first-order model also follow as a by-product. We also reformulate the continuity-type equation, which is derived from the first-order model, as an integro-differential equation by employing the pseudo-inverse of the accumulative particle distribution. By making use of a modified $p$-Wasserstein distance, we provide the contractivity estimate for absolutely continuous solutions of the continuum equation. 
\end{abstract}

\maketitle \centerline{\date}

\tableofcontents

%
%
%
%
\section{Introduction}\label{sec:1}
\setcounter{equation}{0}

The study of collective behaviors for multi-agent systems have recently attracted the interest of many researchers in many different scientific disciplines such as applied mathematics, biology, and control community due to their biological and engineering applications \cite{A82,ADKMZ08,CKPP19,LPLSFD07,PGE09,TBL06}. In the current work, we are interested in the celebrated Cucker--Smale (in short CS) model \cite{CS07} in one dimension, which is a second-order particle system, proposed to describe the velocity alignment behaviors. More precisely, let $x_i(t)$ and $v_i(t)$ be the position and velocity of $i$-th agent at time $t \geq 0$, respectively. Then the CS model reads as
\begin{align}\label{p_CS}
\begin{aligned}
\frac{dx_i(t)}{dt} &= v_i(t), \quad i=1,\dots, N, \quad t >0,\cr
\frac{dv_i(t)}{dt} &= \frac1N \sum_{j=1}^N \psi(x_j(t) - x_i(t))(v_j(t) - v_i(t)),
\end{aligned}
\end{align}
with the initial data
\[
(x_i(0), v_i(0)) =: (x_i^0, v_i^0), \quad  i=1,\dots, N.
\]
Here the function $\psi : \R \to \R_+:= (0,\infty)$ is called the communication weight, which is given as nonnegative and nonincreasing function and $N$ is the number of agents. In \cite{CS07}, the regular communication weight
\bq\label{comm_reg}
\psi(x) = \frac{1}{(1 + |x|^2)^{\alpha/2}} \quad \mbox{with} \quad \alpha > 0
\eq
is considered, and the sufficient frameworks for the velocity alignment behaviors are investigated. 

The CS model \eqref{p_CS} and its variants have been studied extensively, and some previous works can be summarized as follows. The large time behaviors of CS model including the emergences of mono-cluster and multi-cluster are obtained in \cite{H-K-P-Z,H-K-Z,HT08}. The mean field limits of the CS-type models with regular/singular communication weights are studied in \cite{CCHS19, C-F-R-T,D-F-T,HKZ18,HL09,MP}. Other generalization and variants of CS model including noises, normalization weights, temperatures, and time delays, are discussed in \cite{CL18, C-M,H-R,M-T}. We refer to \cite{CCP17, CHL17, MMPZ19} and references therein for recent surveys on the CS-type flocking models. On the other hand, for a very large number of agents, the particle system \eqref{p_CS} is computationally complicated, and thus it is better to derive a corresponding continuum model of the dynamics by introducing a distribution function. At the formal level, we can derive the following Vlasov-type equation from the CS model \eqref{p_CS} as $N \to \infty$:
\bq\label{k_CS}
\begin{aligned}
&\pa_t F + v \pa_x F + \pa_v  \lt( \mathcal{A}(F) F \rt) = 0, \quad (x,v) \in \R \times \R, \quad t > 0,\\
&\mathcal{A}(F)(x,v,t) = \iint_{\R \times \R} \psi(x-y)(w-v) F(y,w,t)\,dydw,
\end{aligned}
\eq
where $F = F(x,v,t)$ denotes the one-particle distribution function at the position-velocity $(x,v) \in \R \times \R$ at time $t$, and $\mathcal{A}(F)$ is the velocity alignment force.

In the present work, we deal with a singular communication weight of the form
\bq\label{comm_sing}
\psi(x) = \frac{1}{|x|^\beta} \quad \mbox{with} \quad \beta > 0.
\eq
Note the integral $\mathcal{A}(f)(x,v,t)$ is not well defined when $\beta\geq 1$ in one dimensional case, thus we have to require $\beta<1$ in this paper. It is worth noticing that the original CS model \cite{CS07}, i.e., the system \eqref{p_CS} with \eqref{comm_reg}, does not take into account the collision avoidance between agents. For the singular CS model, i.e., the system \eqref{p_CS} with \eqref{comm_sing}, in \cite{CCMP17}, a critical value of the exponent $\beta$ in $\psi$ leading to global existence and finite-time collision between agents. More precisely, if $\beta \geq 1$, there is no collision between agents in finite time. On the other hand, we can construct an initial configuration for finite-time collision for $\beta < 1$, see \cite{P14}. This enables us to have the collision avoidance property, not by adding additional short-range repulsive forces. Despite this advantage of considering the singular communication weight, it generates a lot of difficulties in analysis due to the singularity. So far, there are only a few works on the singular CS model. The  collision avoidance and large time behaviors of the singular CS model are investigated in \cite{CCMP17,MMPZ19,ZZ20}. In \cite{MP,P14}, global-in-time existence of measure-valued solutions to the kinetic CS equation, see Section \ref{sec:def_mea} for the definition, is established and the global well-posedness of the singular CS particle system are obtained, respectively. In particular, the singular CS model with $\beta<1/2$ is studied in \cite{MP}, and constructed the measure-valued solution through atomic approximation with different mass. However, this limit process is a little bit different from the mean-field limit of the original CS model. Moreover, the authors in \cite{CCM14} applied the contraction mapping criteria to construct the local well-posedness of general CS model with singular communication weights and nonlocal velocity couplings. This type of construction depends on small time to yield a contraction map, and therefore it is difficult to extend the local-in-time solution to global one. We also refer to \cite{AC} for the well-posedness of continuity equation with non-smooth velocity.

In the current work, we will rigorously prove the uniform-in-time mean-field limit from \eqref{p_CS} to \eqref{k_CS}, and then establish global-in-time well posedness of measure-valued solutions to the kinetic equation \eqref{k_CS} in the regime $\beta \in (0,1)$. In order to overcome the difficulties mentioned above, inspired by \cite{ZZ20}, we first reformulate the particle system \eqref{p_CS} as a first-order swarming model. To be more specific, for a weakly singular regime $\beta \in (0,1)$, a suitable change of variable to the second-order system \eqref{p_CS} is proposed in \cite{ZZ20}, and an equivalent first-order system  is successfully derived:
\bq\label{D1}
\frac{dx_i(t)}{dt} = \omega_i + \frac1N \sum_{j=1}^N \Psi(x_j(t) - x_i(t)), \quad i=1,\dots,N, \quad t > 0,
\eq
where $\omega_i$ is the natural velocity of $i$-th agent and the interaction function $\Psi$ is the anti-derivative of $\psi$. More precisely, we have
\bq\label{eq_Psi}
\omega_i=v_i(0)-\frac1N \sum_{j=1}^N \Psi(x_j(0) - x_i(0)),\quad \Psi(x) := \int_0^x \psi(r)\,dr = \frac{sgn(x)}{1-\beta} \frac{1}{|x|^{\beta-1}}, \quad \beta \in (0,1).
\eq
Note that the communication weight $\psi$ is singular for the second-order CS model, however the interaction function $\Psi$ is regular for first-order reformulation \eqref{D1}. Therefore, we can apply classical methods to the first order model \eqref{D1}, and then further apply the equivalence between \eqref{p_CS} and \eqref{D1} to obtain the desired results for the original model \eqref{p_CS}. Moreover, even though the particle system \eqref{D1} plays a role as an intermediate system, it resembles the celebrate Kuramoto model for the synchronization phenomena \cite{Ku84} and exhibits a consensus behavior of agents. We refer to \cite{ZZ20} for the global existence and uniqueness of classical solutions and its large time behaviors, see also Section \ref{ssec:lt}.

Our main results are three folds. First, in Section \ref{sec:3}, we begin with the stability estimate for the first-order particle system \eqref{D1}. By using the large time behavior estimate of solutions and a modulated $\ell_p$ distance, we show the uniform-in-time stability of solutions to \eqref{D1}. In particular, if two solutions of \eqref{D1} have the same natural velocity $\omega = (\omega_1,\dots,\omega_N)$, then we have the contractivity estimate for solutions, see Theorem \ref{L4.1} in Section \ref{ssec:first} for details.  We take a full advantage of this stability estimate for the first-order model and extend it to the singular CS particle system \eqref{p_CS} in Section \ref{ssec:CS}. This is one of our main contributions in this work.  

Secondly, in Section \ref{sec:measure}, we analyze the kinetic equation \eqref{k_CS}. By employing the particle-in-cell method and $p$-Wasserstein distance, see Definition \ref{D2.2}, we provide a quantitative error estimate between the particle system \eqref{p_CS} and the kinetic equation \eqref{k_CS}. We consider an empirical measure associated to the particle system \eqref{p_CS} and apply the uniform-in-time stability estimate for the particle system. This yields the uniform-in-time Cauchy estimate in $p$-Wasserstein distance for that empirical measures, and thus we construct the global-in-time measure-valued solutions to the kinetic equation \eqref{k_CS}. We emphasize that the quantitative bound is independent of time $t$. As a direct consequence, we rigorously prove the uniform-in-time mean-field limit of the particle system \eqref{p_CS}. It is also worth noticing that the our strategy provides the uniqueness of measure-valued solutions in the sense of mean-field limits, i.e., the measure-valued solutions, approximated by the empirical measures associated to the particle system, uniquely exist. Moreover, we also establish the uniform-in-time mean-field limit from the first-order system \eqref{D1} to the following continuity-type equation:
\bq\label{main_eq}
\pa_t f + \pa_x ((\omega - \Psi \star \rho)f) = 0.
\eq
Here $f = f(x,\omega,t)$ is the one-particle distribution function at a phase point $(x,v) \in \R \times \R$ at time $t$, and $\rho$ is the first marginal of $f$, i.e., $\rho = \int_\R f\,d\omega$. The global-in-time existence and uniqueness of measure-valued solutions are also presented. 

Finally, in Section \ref{sec:contr}, we will study the particle system \eqref{D1} and the corresponding continuum equation \eqref{main_eq}. As mentioned before, these two equations are interesting models themselves. In particular, the equation \eqref{main_eq} has a gradient flow structure, see Section \ref{ssec:grad}. In Section \ref{sec:contr}, we provide the contractivity of solutions to \eqref{main_eq} by using the method of optimal mass transport. We rewrite the continuity-type equation in terms of the pseudo-inverse of the cumulative distribution and introduce a modified $p$-Wasserstein distance to estimate the error between two absolutely continuous solutions to the equation \eqref{main_eq}. 

The rest of the paper will be organized as follows. In Section \ref{sec:2}, we will provide some preliminary concepts and previous results which will be used in the later sections. In Section \ref{sec:3}, we construct the uniform-in-time stability of the particle system \eqref{p_CS} and \eqref{D1}. In Section \ref{sec:measure}, we will establish the uniform-in-time mean-field limit and the global-in-time well posedness for the kinetic equations \eqref{k_CS} and \eqref{main_eq}. In Section \ref{sec:contr}, we will construct the contractivity of solutions to the continuity equation \eqref{main_eq}, which immediately yields the uniqueness of solutions to  \eqref{k_CS} and \eqref{main_eq}. Finally, Section \ref{sec:6} will be contributed as a conclusion. 

Throughout this paper, we denote by $C$ a generic, no necessarily identical, positive constant independent of $N$ and $t$.\newline

%
%
%
%
\section{Preliminaries}\label{sec:2}
\setcounter{equation}{0}

In this section, we will present some properties and estimates on the first-order particle system \eqref{D1} and its continuum version \eqref{main_eq}. Then we will discusse the general notions of measure-valued solutions to continuity-type equations. Finally, we will provide a brief explanation on the relation between kinetic CS model \eqref{k_CS} and the continuity-type equation \eqref{main_eq}.

\subsection{A gradient flow formulation}\label{ssec:grad}
In this part, we show that the continuity-type equation \eqref{main_eq} has a gradient flow structure. Note that if we set
\[
K(x) := \int_0^x \Psi(y)\,dy = \frac{1}{(2-\beta)(1-\beta)}|x|^{2-\beta},
\]
then the force field in \eqref{main_eq} can be rewritten as 
\[
\omega - \Psi \star \rho = \pa_x (\omega x - K \star \rho) = -\pa_x \delta_f \mathcal{F}, 
\]
where $\delta_f \mathcal{F} = \delta \mathcal{F}/\delta f$ is the functional derivative of a free energy $\mathcal{F}$ given by
\[
\mathcal{F}(f):=   \frac12\intrr K(x-y)f(x,\omega)f(y,w)\,dxd\omega dydw - \intrr \omega x f\,dxd\omega.
\]
This shows that our continuum equation \eqref{main_eq} can be reformulated as
\[
\pa_t f = \pa_x ((\pa_x \delta_f \mathcal{F})f).
\]
Furthermore, from this reformulation, we can easily find
\[
\frac{d}{dt} \mathcal{F}(f) = \intrr (\delta_f \mathcal{F}) \pa_t f\,dxd\omega = -\intrr |\pa_x \delta_f \mathcal{F}|^2 f\,dxd\omega = -\intrr |\omega - \Psi \star \rho|^2 f\,dxd\omega.
\]
\begin{remark}\label{rmk_free}
Let us assume that the $f$ is compactly supported in both $x$ and $\omega$ and we can easily find that there exists a positive constant $c_0>0$ such that $\mathcal{F}(f) \geq - c_0$. This asserts the uniform-in-time integrability of the second moment of the velocity field
\[
\int_0^\infty\intrr |\omega - \Psi \star \rho|^2 f\,dxd\omega dt < \infty.
\]
\end{remark}

\subsection{Energy estimates}\label{ssec:energy} In this part, we provide some a priori energy estimates for the continuity-type equation \eqref{main_eq}. 

\begin{lemma}\label{lem_energy} Let $f$ be a solution to the equation \eqref{main_eq} with sufficient integrability. Then we have
\begin{itemize}
\item[(i)] The total mass is conserved:
\[
\intrr f(x,\omega,t)\,dxd\omega = \intrr f_0(x,\omega)\,dxd\omega.
\]
\item[(ii)] The $p$-th moments in $\omega$ is conserved:
\[
\intrr |\omega|^p f(x,\omega,t)\,dxd\omega = \intrr |\omega|^p f_0(x,\omega)\,dxd\omega,\quad p \geq 1.
\]
\item[(iii)] The second moment of the velocity field is nonincreasing:
\end{itemize}
\bq\label{eq_free}
\frac{d}{dt}\mathcal{E}(f) + \mathcal{D}(f) = 0,
\eq
where the energy functional $\mathcal{E}(f)$ and its dissipation rate $\mathcal{D}(f)$ are given by
\[
\mathcal{E}(f) := \frac12 \intrr | \omega - (\Psi \star \rho)(x,t)|^2 f(x,\omega,t)\,dxd\omega
\]
and
$$\begin{aligned}
\mathcal{D}(f) &:= \intrr (\psi \star \rho)|\omega - \Psi \star \rho|^2 f\,dxd\omega\cr
& - \intrr (u(x) - (\Psi \star \rho)(x)) \psi(x-y) (u(y) -  (\Psi \star \rho)(y)) \rho(x)\rho(y)\,dxdy,
\end{aligned}$$
respectively.
\end{lemma}
\begin{proof} A straightforward computation gives
$$\begin{aligned}
&\frac12 \frac{d}{dt}\intrr | \omega - \Psi \star \rho|^2 f(x,\omega,t)\,dxd\omega\cr
&\quad = - \intrr  (\omega - \Psi \star \rho) (\Psi \star \pa_t \rho) f\,dxd\omega + \frac12 \intrr | \omega - \Psi \star \rho|^2 \pa_t f \,dxd\omega\cr
&\quad =: I_1 + I_2.
\end{aligned}$$
For the estimate of $I_1$, we denote $(\rho u)(y) := \intr \omega f\,d\omega$ and then we have the following estimates,
$$\begin{aligned}
\Psi \star \pa_t \rho &= \intr \Psi(x-y) \pa_t \rho (y)\,dy\cr
&=-\intrr \Psi(x-y) \pa_y \lt((\omega - (\Psi \star \rho)(y))f(y,\omega)\rt)dyd\omega\cr
&=-\intrr \psi(x-y) (\omega -  (\Psi \star \rho)(y))f(y,\omega)\,dyd\omega\cr
&=-\intr \psi(x-y) (u(y) -  (\Psi \star \rho)(y)) \rho(y)\,dy.
\end{aligned}$$
This deduces that 
\begin{align}\label{I_1}
\begin{aligned}
I_1 &= -\intr (u(x) - (\Psi \star \rho)(x)) (\Psi \star \pa_t \rho) \rho(x)\,dx\cr
&= \intrr (u(x) - (\Psi \star \rho)(x)) \psi(x-y) (u(y) -  (\Psi \star \rho)(y)) \rho(x)\rho(y)\,dxdy.
\end{aligned}
\end{align}
Next for $I_2$, we have the following estimates,
$$\begin{aligned}
I_2 &= - \frac12 \intrr | \omega - \Psi \star \rho|^2 \pa_x ((\omega - \Psi \star \rho)f) \,dxd\omega\cr
&= - \intrr (\omega - \Psi \star \rho)(\psi \star\rho)  (\omega - \Psi \star \rho)f \,dxd\omega\cr
&= - \intrr (\psi \star \rho)|\omega - \Psi \star \rho|^2 f\,dxd\omega.
\end{aligned}$$
This combined with \eqref{I_1} asserts the free energy estimate \eqref{eq_free}. It remains to show the nonnegativity of the dissipation rate $\mathcal{D}(f)$. For this, it suffices to get $I_1 \leq - I_2$. By using H\"older's inequality, we find
\[
I_1 \leq \intrr (\psi \star \rho)|u - \Psi \star \rho|^2 \rho\,dx.
\]
On the other hand, we obtain
\[
\rho |u|^2 \leq \frac{|\rho u|^2}{\rho} = \frac{\lt|\intr \omega f\,d\omega \rt|^2}{\rho} \leq \intr |\omega|^2 f\,d\omega,
\]
and this yields
\[
|u - \Psi \star \rho|^2 \rho \leq \intr |\omega - \Psi \star \rho|^2 f\,d\omega.
\]
Thus we have $I_1 \leq -I_2$, and subsequently, this concludes $\mathcal{D}(f) = - (I_1 + I_2) \geq 0$.
\end{proof}

\subsection{Large time behavior for the particle system \eqref{D1}}\label{ssec:lt} In this part, we recall the results on the global existence and uniqueness of solutions to the first-order particle system \eqref{D1} and its large time behavior. 

For notational simplicity, we denote by $X(t) = (x_1(t), \cdots, x_N(t))$, where $x_i(t)$ is the position of $i$-th particle at time $t \geq 0$, and $X_0 = X(0)$. We also introduce diameters of position and natural velocity configurations:
\[
D_x(t) := \max_{1 \leq i,j \leq N} |x_i(t) - x_j(t)| \quad \mbox{and} \quad D_{\omega}:=  \max_{1 \leq i,j \leq N} |\omega_i - \omega_j|.
\]

\begin{theorem}[\cite{ZZ20}]\label{T3.1}
There exists a unique $\mc^1$ solution to the first-order CS model \eqref{D1} with $0<\beta<1$. Moreover, assume that the natural velocities and initial data satisfy the properties below,
\[\sum_{i=1}^N x_i^0 = 0\quad \mbox{and} \quad \sum_{i=1}^N \omega_i = 0.\]
Then, the following assertions hold:
\begin{itemize}
\item[(i)]
There exist positive constants $C_0:= \max \left\{D_x(0)\,, \Psi^{-1}(D_{\omega}) \right\}$ such that
\begin{equation*}
|x_i(t)-x_j(t)|\leq C_0,\quad 1\leq i\neq j\leq N,\quad t\geq 0. 
\end{equation*}
\item[(ii)]
Unconditional flocking occurs and there exists an equilibrium state $X^\infty=(x_1^\infty,\cdots,x_N^\infty)$ such that 
\[
|X(t)-X^{\infty}|\leq C e^{-\psi(C_0)t},
\]
where $C>0$ is independent of $t$.
\end{itemize}
\end{theorem}
\noindent Note the diameter of any solution $X(t)$ depends only on the diameter of diameter of initial data $X(0)$ and diameter of natural velocities, which is independent of $N$. This is the one of the most important properties of the system \eqref{D1}, and it will be crucially used in the estimates of the uniform-in-time stability and mean-field limit later in Section \ref{sec:3}.

\subsection{Measure-valued solutions}\label{sec:def_mea}
In this subsection, we provide general concepts of measure-valued solutions and related ideas for the following continuity-type equation:
\begin{equation}\label{B1}
\partial_tg +\nabla_{z}\cdot (v[g] g)=0,\quad z\in\mathbb R^d,~~t > 0,
\end{equation}
where $v[g]$ is a $n$-dimensional vector field. Then we have the following definition of measure-valued solution of \eqref{B1}. 
\begin{definition} \label{D4.1}
For $T\in[0,\infty)$, $\mu_t\in L^{\infty}([0,T);\mathcal{P}(\mathbb R^d))$ is a measure-valued solution to \eqref{B1} with initial data $\mu_0\in\mathcal{P}(\mathbb R^d)$ if the following three assertions hold:\newline
\begin{itemize}
\item[(i)] Total mass is normalized: $\langle\mu_t,1\rangle=1$,
\item[(ii)] $\mu$ is weakly continuous in $t$:
\[ \langle\mu_t,\phi\rangle~\mbox{is continuous in $t$}, \quad \forall\,\phi(z) \in C_0^1(\mathbb R^d),
\]
\item[(iii)] $\mu$ satisfies the equation \eqref{B1} in the sense of distributions: 
\begin{equation*}
 \langle\mu_t, \varphi(\cdot,t)\rangle-\langle \mu_0, \varphi(\cdot,0)\rangle=
\int_0^t\langle\mu_s,\partial_s \varphi + v[\mu]\cdot \nabla_z \varphi  \rangle \,  ds, \quad \forall\,\varphi \in C^1_0(\mathbb R^d \times [0, T))
\end{equation*}
\end{itemize}
\end{definition}

\begin{remark}\label{R2.1}
There are two claims about the measure-valued solution. 
\begin{itemize}
\item[(i)] Note that if we let $z=(x,v)$, $g(z,t)=F(x,v,t)$ and the vector field to be $v[g]=(v,\mathcal{A}(F))$, then we obtain the definition of the measure-valued solution to \eqref{k_CS}. On the other hand, if we let $z=(x,\omega)$, $g(z,t)=f(x,\omega,t)$ and the vector field to be $v[g]=(w - \Psi \star \rho,0)$ in \eqref{B1}, then we have the definition of the measure-valued solution to \eqref{main_eq}.
\item[(ii)] Let $\mu^N_t$ be an empirical measure associated to the particle system \eqref{p_CS}, i.e.,
\[
\mu^N_t := \frac1N\sum_{i=1}^{N}\delta_{(x_i(t),v_i(t))},
\]
where $(x_i(t),v_i(t))$ is a solution to \eqref{p_CS}. Then it is easy to check that $\mu^N_t$ is a measure-valued solution to \eqref{k_CS}. Similarly, the empirical measure 
\bq\label{eq_mu}
\sigma^N_t := \frac1N\sum_{i=1}^{N}\delta_{(x_i(t),\omega_i)}
\eq 
is a measure-valued solution to \eqref{main_eq} provided $x_i(t)$ is a solution to \eqref{D1} with natural velocities $\omega_i$.
\end{itemize}
\end{remark}
\noindent Finally, we equip a proper metric to the probability measure space ${\mathcal P}(\bbr^{2d})$ to measure the distance between two measures, and then we will introduce the concept of local-in-time mean-field limit. In fact, we can endow $p$-Wasserstein distance $W_p$ in the probability space ${\mathcal P}_p(\bbr^{2d})$ which represents a collection of all probability measures with finite $p$-th moment, i.e., $\langle\mu, |z|^p \rangle<+\infty$.
\begin{definition}[\cite{AGS05, N84, V08}]\label{D2.2}
\begin{enumerate}
\item[(i)] Let $p\in [1,\infty]$. Then $p$-Wasserstein distance $W_p(\mu, \nu)$ is defined for any $\mu,\nu\in {\mathcal P}_p(\bbr^{2d})$ as
\[W_p(\mu,\nu):= \inf_{\gamma\in\Gamma(\mu,\nu)} \lt( \iint_{\bbr^{2d}\times \bbr^{2d}} |z-z^*|^p\,\gamma(dzdz^*) \rt)^{\frac{1}{p}}, \]
where $\Gamma(\mu,\nu)$ denotes the collection of all probability measures on $\bbr^{2d}\times \bbr^{2d}$ with marginals $\mu$ and $\nu$.

\item[(ii)] For any $T \in (0, \infty]$, the kinetic equation \eqref{main_eq} is derivable from the particle model \eqref{D1} in $[0,T)$, or equivalent to say the mean-field limit from the particle system \eqref{D1} to the kinetic equation \eqref{main_eq}, which is valid in $[0,T)$, if for every solution $\mu_t$ of the kinetic equation \eqref{main_eq} with initial data $\mu_0$, the following condition holds: for some $p \in [1, \infty]$ and $t \in [0, T)$,
\[ \lim_{N\rightarrow +\infty}W_p(\mu_0^N,\mu_0)=0 \quad \Longleftrightarrow \quad  \lim_{N\rightarrow +\infty}W_p(\mu_t^N,\mu_t)=0,\]
where $\mu^N_t$ is a measure-valued solution of the particle system \eqref{D1} with initial data $\mu^N_0$. 
\end{enumerate}
\end{definition}

\subsection{Relation between \eqref{k_CS}\label{ssec:2.5} and \eqref{main_eq}} Define a mapping $\Gamma: \R \times \R \times \R_+ \to \R \times \R \times \R_+$ by $\Gamma(x,\omega,t) = (x, \omega - \Psi \star \rho,t)$. Then $\Gamma$ is differentiable with the following Jaccobian matrix,
\[
D\Gamma = \left( \begin{array}{ccc} 1 & -\psi\star\rho & 0 \\ 0 & 1 & 0 \\ 0 & \psi \star  (\rho u)& 1 \end{array} \right),
\]
where $\rho u = \int_\R \omega f\,d\omega$, and this gives $det(D\Gamma) = 1 \neq 0$. Note that for $\varphi = \varphi(x,v,t) \in \mc^1_0(\R \times \R \times [0,T])$, we apply the chain rule to obtain that 
\begin{align*}
&\pa_t (\varphi \circ \Gamma) = (\pa_t \varphi) \circ \Gamma + (\psi \star (\rho u)) (\pa_v \varphi) \circ \Gamma,\\ 
&\pa_x  (\varphi \circ \Gamma)  = (\pa_x \varphi) \circ \Gamma - (\psi \star \rho) (\pa_v \varphi) \circ \Gamma.
\end{align*}
We let $\mu_t$ to be the measure valued solution to \eqref{main_eq}. Then  we consider a test function $\varphi \circ \Gamma$ in Definition \ref{D4.1} (3) and obtain that 
$$\begin{aligned}
&\intrr \varphi(\Gamma(x,\omega,t)) \,\mu_t(dxd\omega) - \intrr \varphi(\Gamma(x,\omega,0)) \,\mu_0(dxd\omega)\cr
&\quad = \int_0^t \intrr \lt(\pa_s (\varphi \circ \Gamma) + (\omega - \Psi \star \rho) \pa_x (\varphi \circ \Gamma)\rt)\, \mu_s(dxd\omega)ds \cr
&\quad = \int_0^t \intrr (\pa_s \varphi) \circ \Gamma + (\psi \star (\rho u)) (\pa_v \varphi) \circ \Gamma \cr
&\hspace{3cm} + (\omega - \Psi \star \rho) ((\pa_x \varphi) \circ \Gamma - (\psi \star \rho) (\pa_v \varphi) \circ \Gamma)\, \mu_s(dxd\omega)ds\cr
&\quad = \int_0^t \intrr \Big[(\pa_s \varphi) +(\omega - \Psi \star \rho) (\pa_x \varphi)  + \left(\psi \star (\rho u) -(\omega - \Psi \star \rho) (\psi \star \rho)\right) (\pa_v \varphi)\Big] \circ \Gamma\, \mu_s(dxd\omega)ds\cr
&\quad = \int_0^t \intrr \lt(\pa_s \varphi  + v \pa_x \varphi + (\psi \star (\rho u) - v(\psi\star\rho)) \pa_v \varphi   \rt) (\Gamma \# \mu_s)(dxdv)ds.
\end{aligned}$$
This asserts that $\Gamma(\cdot, \cdot,t) \# \mu_t$ satisfies the kinetic equation \eqref{k_CS} in the sense of distributions.\newline

\section{Uniform-in-time stability estimates}\label{sec:3}
\setcounter{equation}{0}

In this section, we present our result on the uniform-in-time stability estimate for the singular CS particle system \eqref{p_CS}. For this, we first estimate the uniform-in-time stability of solutions for the first-order particle system \eqref{D1}. We then extend it to the system \eqref{p_CS}. 

\subsection{Uniform-in-time stability for the first-order system \eqref{D1}}\label{ssec:first}
For two solutions $X$ and $Y$ to the equation \eqref{D1}, let us define a modulated $\ell_p$ distance between $X$ and $Y$ as
\[
\mathcal{X}(t):=\left(\frac{1}{N}\sum_{i=1}^N((x_i(t)-x_c-\omega_ct)-(y_i(t)-y_c-\omega_ct))^p\right)^{\frac{1}{p}},
\]
where the averages quantities $x_c,\, y_c, \,\omega_c$, and $w_c$ are given below,
\[
x_c = \frac1N \sum_{i=1}^N x^0_i, \quad y_c = \frac1N \sum_{i=1}^Ny^0_i, \quad \omega_c = \frac1N \sum_{i=1}^N \omega_i, \quad \mbox{and} \quad w_c = \frac1N \sum_{i=1}^N w_i.
\]
Then the uniform-in-time stability of solutions for the system \eqref{D1} is as follows.

\begin{theorem}\label{L4.1} Let $X$ and $Y$ be global solutions to the equation \eqref{D1} with initial data $X_0$ and $Y_0$, natural velocities $(\omega_i)$ and $(w_i)$, respectively. Then we have 
\[
\mathcal{X}(t)\leq e^{-\psi(2D_0)t}\mathcal{X}(0)+\frac{\mathcal{U}}{\psi(2D_0)}  \quad \forall\, t \geq 0,
\]
where the constant $D_0 > 0$ and $\mathcal{U}$ are given as follows,
\[
D_0 :=\max\lt\{D_x(0),D_y(0), \Psi^{-1}(D_{\nu}), \Psi^{-1}(D_{\omega})\rt\},\quad \mathcal{U}:=\left(\frac{1}{N}\sum_{i=1}^N|(\omega_i-\omega_c)-(w_i-w_c)|^p\right)^{\frac{1}{p}}.
\]
\end{theorem}
\begin{proof}
Without loss of generality, we may assume the mean zero property of the spatial variables and natural velocities, i.e., 
\[\sum_{i=1}^Nx_i^0=\sum_{i=1}^Ny_i^0=\sum_{i=1}^N\omega_i^0=\sum_{i=1}^N w_i^0=0.\]
Then, according to previous analysis in Theorem \ref{T3.1}, the diameter of both $X$ and $Y$ are bounded by $D_0$. Then, the difference between $x_i$ and $y_i$ satisfies the following equation,
\begin{align*}
\begin{aligned}
\frac{1}{p}\frac{d}{dt}|x_i-y_i|^p&=|x_i-y_i|^{p-2}(x_i-y_i)(\omega_i-w_i)\\
&\quad + |x_i-y_i|^{p-2}(x_i-y_i)\left(\frac1N \sum_{j=1}^N \Psi(x_j - x_i)- \frac1N \sum_{j=1}^N \Psi(y_j - y_i)\right).
\end{aligned}
\end{align*}
Then we add all equations for each $i$-th particle and obtain for any integer $p$ that,
\begin{align}\label{D4.2}
\begin{aligned}
&\frac{1}{p}\frac{d}{dt}\sum_{i=1}^N|x_i-y_i|^p\\
&\quad =\sum_{i=1}^N |x_i-y_i|^{p-2}(x_i-y_i)(\omega_i-w_i)\\
&\qquad + \sum_{i=1}^N |x_i-y_i|^{p-2}(x_i-y_i)\left(\frac1N \sum_{j=1}^N \Psi(x_j - x_i)- \frac1N \sum_{j=1}^N \Psi(y_j - y_i)\right)\\
&\quad \leq \lt(\sum_{i=1}^N|x_i-y_i|^p\rt)^{\frac{p-1}{p}}\lt(\sum_{i=1}^N|\omega_i- w_i|^p\rt)^{\frac{1}{p}}\\
&\qquad + \sum_{i=1}^N |x_i-y_i|^{p-2}(x_i-y_i)\left(\frac1N \sum_{j=1}^N \Psi(x_j - x_i)- \frac1N \sum_{j=1}^N \Psi(y_j - y_i)\right).
\end{aligned}
\end{align}
For the last term in \eqref{D4.2}, we applied the odd property of $\Psi$ and the symmetry property of the summation to obtain that 
\begin{align}\label{D4.3}
\begin{aligned}
&\sum_{i=1}^N |x_i-y_i|^{p-2}(x_i-y_i)\left(\frac1N \sum_{j=1}^N \Psi(x_j - x_i)- \frac1N \sum_{j=1}^N \Psi(y_j - y_i)\right)\\
&\quad =-\frac{1}{2N}\sum_{i,j}\left(|x_j-y_j|^{p-2}(x_j-y_j)-|x_i-y_i|^{p-2}(x_i-y_i)\right)\left( \Psi(x_j - x_i)-  \Psi(y_j - y_i)\right).
\end{aligned}
\end{align}
It is obviously that $|s|^ps$ is an increasing function of $s\in\bbr$. Now we assume $(x_j - x_i)\geq(y_j - y_i)$, then we immediately obtain that $(x_j-y_j)\geq(x_i-y_i)$. Therefore, we combine the increasing property of $|s|^ps$ to conclude that \eqref{D4.3} is always non-positive. Now we apply the increasing property of $\Psi$, Lagrangian remainder theory, the uniform upper bound of spatial diameter $D_0$ and the mean zero assumption to obtain 
\begin{align}\label{D4.4}
\begin{aligned}
&\sum_{i=1}^N  |x_i-y_i|^{p-2}(x_i-y_i)\left(\frac1N \sum_{j=1}^N \Psi(x_j - x_i)- \frac1N \sum_{j=1}^N \Psi(y_j - y_i)\right)\\
&\quad =-\frac{1}{2N}\sum_{i,j}\left(|x_j-y_j|^{p-2}(x_j-y_j)-|x_i-y_i|^{p-2}(x_i-y_i)\right)\left( \Psi(x_j - x_i)-  \Psi(y_j - y_i)\right)\\
&\quad \leq -\frac{1}{2N}\sum_{i,j}\left(|x_j-y_j|^{p-2}(x_j-y_j)-|x_i-y_i|^{p-2}(x_i-y_i)\right)(x_j-x_i-(y_j - y_i))\psi(2D_0)\\
&\quad \leq -\psi(2D_0)\sum_{i=1}^N|x_i-y_i|^p.
\end{aligned}
\end{align}
For the other case $(x_j - x_i)\leq(y_j - y_i)$, we can use similar analysis to obtain the same estimate. Thus we apply \eqref{D4.2} and \eqref{D4.4} to imply that 

\begin{align}\label{D4.5}
\begin{aligned}
&\frac{1}{p}\frac{d}{dt}\sum_{i=1}^N|x_i-y_i|^p\leq \lt(\sum_{i=1}^N|x_i-y_i|^p\rt)^{\frac{p-1}{p}}\lt(\sum_{i=1}^N|\omega_i-w_i|^p\rt)^{\frac{1}{p}}-\psi(2D_0)\sum_{i=1}^N|x_i-y_i|^p.
\end{aligned}
\end{align}
Now, we define use the definition of $\mathcal{X}$ and $\mathcal{U}$ and apply \eqref{D4.5} to obtain 
\begin{equation*}
\frac{d}{dt}\mathcal{X}\leq \mathcal{U}-\psi(2D_0)\mathcal{X}.
\end{equation*}
Then, by simple calculation, we obtain the estimate of $\mathcal{X}$ as 
\begin{equation*}
\mathcal{X}(t)\leq e^{-\psi(2D_0)t}\mathcal{X}(0)+\frac{\mathcal{U}}{\psi(2D_0)}.
\end{equation*}
Now, for general initial data and natural velocity, we notice the conservation law of the mean position and mean natural velocity. Therefore we make change of variable as follows,
\[\bar{x}=x-x_c,\quad \bar{y}=y-y_c,\quad \bar{\omega}=\omega-\omega_c\quad \mbox{and} \quad \bar{w}=w-w_c.\]
Then for $\bar{x}$, $\bar{y}$, $\bar{\omega}$, $\bar{w}$, we can apply above analysis to obtain the stability and finish the proof of the lemma.
\end{proof}

\begin{remark} If two solutions $X$ and $Y$ have the same natural velocities, then the modulated $\ell_p$ distance of natural velocities $\mathcal{U}$ appeared in Theorem \ref{L4.1} becomes zero. Thus Theorem \ref{L4.1} provides the following exponential contractivity estimate of solutions:
\[
\mathcal{X}(t)\leq e^{-\psi(2D_0)t}\mathcal{X}(0)
\]
for $t \geq 0$. Later in Section \ref{sec:contr}, we present the contractivity estimate of solutions for the continuum equation \eqref{main_eq}. 
\end{remark}

\subsection{Uniform-in-time stability for the singular CS model \eqref{p_CS}}\label{ssec:CS}
In this part, we investigate the uniform-in-time stability estimate for the singular CS particle system \eqref{p_CS}. We first begin with the auxiliary lemma which plays a crucial role in estimating the stability of solutions to \eqref{p_CS}.

\begin{lemma}\label{LC1}
For any $i,j \in \{1,\dots, N\}$ and the interaction $\Psi$ appeared in \eqref{eq_Psi}, we have 
\begin{equation}\label{3.11}
|\Psi(x_j-x_i)-\Psi(\bar{x}_j-\bar{x}_i)|\leq2\left(|\Psi(x_j-\bar{x}_j)|+|\Psi(x_i-\bar{x}_i)|\right).
\end{equation}
\end{lemma}
\begin{proof}
We will split the proof into two cases.

\noindent $\bullet$ (Case 1) $sgn(x_j-x_i)=sgn(\bar{x}_j-\bar{x}_i)$. In this case, without loss of generality, we may assume 
\[
(x_j-x_i)\geq (\bar{x}_j-\bar{x}_i)\geq 0.
\] 
Then we apply the concavity of $\Psi$ to obtain 
\[(x_j-\bar{x}_j)\geq (x_i-\bar{x}_i)\quad \mbox{and} \quad |\Psi(x_j-x_i)-\Psi(\bar{x}_j-\bar{x}_i)|\leq|\Psi(x_j-x_i-(\bar{x}_j-\bar{x}_i))|.\]
Now, if 
\[
(x_j-\bar{x}_j)\leq 0 \quad \mbox{or} \quad (x_i-\bar{x}_i)\geq 0, 
\]
we immediately find
\begin{equation}\label{3.12}
|\Psi(x_j-x_i-(\bar{x}_j-\bar{x}_i))|\leq \max\left\{|\Psi(x_j-\bar{x}_j)|,|\Psi(x_i-\bar{x}_i)|\right\}.
\end{equation}
On the other hand, if $(x_j-\bar{x}_j)\geq 0\geq(x_i-\bar{x}_i)$, we apply the concavity of $\Psi$ to have 
\begin{equation}\label{3.13}
|\Psi(x_j-x_i-(\bar{x}_j-\bar{x}_i))|\leq|\Psi(x_j-\bar{x}_j)|+|\Psi(x_i-\bar{x}_i)|.
\end{equation}

\noindent $\bullet$ (Case 2) $sgn(x_j-x_i)\neq sgn(\bar{x}_j-\bar{x}_i)$. In this case, without loss of generality, we may assume 
\[
x_j-x_i\geq 0\geq \bar{x}_j-\bar{x}_i. 
\]
Then we consider six subcases respectively below.

If 
\[
x_j\geq x_i\geq \bar{x}_i\geq \bar{x}_j \quad \mbox{or} \quad x_j\geq \bar{x}_i\geq x_i\geq \bar{x}_j,
\] 
then we apply the increasing property of $|\Psi|$ to get
\begin{equation}\label{3.14}
|\Psi(x_j-x_i)-\Psi(\bar{x}_j-\bar{x}_i)|\leq |\Psi(x_j-x_i)|+|\Psi(\bar{x}_j-\bar{x}_i)|\leq 2|\Psi(x_j-\bar{x}_j)|.
\end{equation}
Similarly, if 
\[
\bar{x}_i\geq x_j\geq \bar{x}_j\geq x_i \quad \mbox{or} \quad \bar{x}_i\geq \bar{x}_j\geq x_j\geq x_i,
\] 
then we have 
\begin{equation}\label{3.15}
|\Psi(x_j-x_i)-\Psi(\bar{x}_j-\bar{x}_i)|\leq |\Psi(x_j-x_i)|+|\Psi(\bar{x}_j-\bar{x}_i)|\leq 2|\Psi(x_i-\bar{x}_i)|.
\end{equation}
Moreover, for the case $x_j\geq \bar{x}_i\geq \bar{x}_j\geq x_i$, we first note that 
\[(x_j-x_i)-(x_j-\bar{x}_j)=(\bar{x}_i-x_i)-(\bar{x}_i-\bar{x}_j)\quad \mbox{and} \quad (x_j-\bar{x}_j)\geq (\bar{x}_i-\bar{x}_j).\]
Then, we use the fact that $\psi$ is decreasing on the positive real line to obtain 
\begin{align*}
&\left[\Psi(\bar{x}_i-x_i)-\Psi(\bar{x}_i-\bar{x}_j)\right]-\left[\Psi(x_j-x_i)-\Psi(x_j-\bar{x}_j)\right]\\
&\quad \leq \int_{\bar{x}_i-\bar{x}_j}^{\bar{x}_i-x_i}\psi(s)\,ds- \int_{x_j-\bar{x}_j}^{x_j-x_i}\psi(s)\,ds= \int_{\bar{x}_i-\bar{x}_j}^{\bar{x}_i-x_i}\lt(\psi(s)ds-\psi(s+x_j-\bar{x}_i)\rt)ds\geq 0,
\end{align*}
which is equivalent to that 
\begin{equation}\label{3.16}
|\Psi(x_j-x_i)|+|\Psi(\bar{x}_i-\bar{x}_j)|\leq |\Psi(x_j-\bar{x}_j)|+ |\Psi(\bar{x}_i-x_i)|.
\end{equation}
Finally, we can apply almost the same argument to the case $\bar{x}_i\geq x_j\geq x_i\geq \bar{x}_j$, and this gives the same inequality as \eqref{3.16}. Thus we combine \eqref{3.12}, \eqref{3.13}, \eqref{3.14}, \eqref{3.15}, and \eqref{3.16} to conclude the desired result \eqref{3.11}. 
\end{proof}

We now provide the uniform-in-time stability estimate for the singular CS model \eqref{p_CS}.

\begin{theorem}\label{L3.2} Let $(X,V)$ and $(\bar{X},\bar{V})$ be global solutions to the equation \eqref{p_CS} with initial data $(X_0,V_0)$ and $(\bar{X}_0,\bar{V}_0)$, respectively. Then there exists a positive constant $C$ independent of $t$ and $N$ such that 
\begin{align*}
&\mathcal{X}(t)\leq C\left(\mathcal{X}(0)+\mathcal{V}(0)+\mathcal{X}(0)^{1-\beta}\right),\\
&\mathcal{V}(t)\leq C\left(\mathcal{V}(0)+\mathcal{X}(0)^{1-\beta}+\left(\mathcal{X}(0)+\mathcal{V}(0)+\mathcal{X}(0)^{1-\beta}\right)^{1-\beta}\right),
\end{align*}
where $\mathcal{V} = \mathcal{V}(t)$ is given by $\mathcal{V}:=\left(\frac{1}{N}\sum_{i=1}^N|(v_i-\bar{v}_i)|^p\right)^{\frac{1}{p}}$.
\end{theorem}
\begin{proof}
According to the derivation of \eqref{D1}, we obtain the relation between the natural velocity $\omega_i$ and the state variables $(x_i,v_i)$. More precisely, we have
\begin{equation}\label{C11}
\omega_i=v_i(t)-\frac1N \sum_{j=1}^N \Psi(x_j(t) - x_i(t)),\quad t\geq 0.
\end{equation}
Then we apply the results in Theorem \ref{L4.1} to have the estimates of the $\mathcal{U}$. In fact, without loss of generality, we can apply the zero mean assumptions, the triangle inequality to obtain that 
\begin{equation}\label{C18}
\begin{aligned}
&\left(\frac{1}{N}\sum_{i=1}^N|\omega_i-\bar{\omega}_i|^p\right)^{\frac{1}{p}}\\
&\quad =\left(\frac{1}{N}\sum_{i=1}^N\lt|v_i(0)-\bar{v}_i(0)-\frac1N \sum_{j=1}^N \Psi(x_j(0) - x_i(0))+\frac1N \sum_{j=1}^N \Psi(\bar{x}_j(0) - \bar{x}_i(0))\rt|^p\right)^{\frac{1}{p}}\\
&\quad \leq \left(\frac{1}{N}\sum_{i=1}^N\left|v_i(0)-\bar{v}_i(0)\right|^p\right)^{\frac{1}{p}} +\left(\frac{1}{N}\sum_{i=1}^N\lt|\frac1N \sum_{j=1}^N \Psi(x_j(0) - x_i(0))-\frac1N \sum_{j=1}^N \Psi(\bar{x}_j(0) - \bar{x}_i(0))\rt|^p\right)^{\frac{1}{p}}.
\end{aligned}
\end{equation}
The first term in \eqref{C18} is simply $\mathcal{V}(0)$, then we apply \eqref{3.11} in Lemma \ref{LC1} and the convexity to obtain that 
\begin{equation}\label{C19}
\begin{aligned}
&\left(\frac{1}{N}\sum_{i=1}^N\Big|\frac1N \sum_{j=1}^N \Psi(x_j(0) - x_i(0))-\frac1N \sum_{j=1}^N \Psi(\bar{x}_j(0) - \bar{x}_i(0))\Big|^p\right)^{\frac{1}{p}}\\
&\quad \leq \left(\frac{1}{N}\sum_{i=1}^N\left(\frac 1 N \sum_{j=1}^N |\Psi(x_j(0) - x_i(0))-\Psi(\bar{x}_j(0) - \bar{x}_i(0))|\right)^p\right)^{\frac{1}{p}}\\
&\quad \leq \left(\frac{1}{N}\sum_{i=1}^N\left(\frac 2 N \sum_{j=1}^N \left(|\Psi(x_j(0) - \bar{x}_j(0))|+|\Psi(x_i(0) - \bar{x}_i(0))|\right)\right)^p\right)^{\frac{1}{p}}\\
&\quad \leq \left(\frac{2^p}{N^2}\sum_{i=1}^N\sum_{j=1}^N\left( |\Psi(x_j(0) - \bar{x}_j(0))|+|\Psi(x_i(0) - \bar{x}_i(0))|\right)^p\right)^{\frac{1}{p}}\\
&\quad \leq \left(\frac{2^{2p-1}}{N^2}\sum_{i=1}^N\sum_{j=1}^N\left( |\Psi(x_j(0) - \bar{x}_j(0))|^p+|\Psi(x_i(0) - \bar{x}_i(0))|^p\right)\right)^{\frac{1}{p}}\\
&\quad \leq \left(\frac{2^{2p}}{N}\sum_{i=1}^N |\Psi(x_j(0) - \bar{x}_j(0))|^p\right)^{\frac{1}{p}}\\
&\quad =\frac{4}{1-\beta} \left(\frac{1}{N}\sum_{i=1}^N |x_j(0) - \bar{x}_j(0)|^{p(1-\beta)}\right)^{\frac{1}{p}},
\end{aligned}
\end{equation}
where we apply the exact form of  $\Psi$ in the last equality above. Therefore, we apply the concavity of $\Psi(x)$ for $x \geq 0$ to obtain 
\begin{equation}\label{C20}
\begin{aligned}
\frac{4}{1-\beta} \left(\frac{1}{N}\sum_{i=1}^N |x_j(0) - \bar{x}_j(0)|^{p(1-\beta)}\right)^{\frac{1}{p}}&\leq \frac{4}{1-\beta} \left(\frac{1}{N}\sum_{i=1}^N |x_j(0) - \bar{x}_j(0)|^{p}\right)^{\frac{(1-\beta)}{p}}= \frac{4}{1-\beta}\mathcal{X}_0^{1-\beta}.
\end{aligned}
\end{equation}
Now, according to \eqref{C11}, similar argument can be applied to any $t>0$. Therefore we combine \eqref{C18}, \eqref{C19} and \eqref{C20} to obtain that 
\begin{equation}\label{C21}
\mathcal{U}\leq \mathcal{V}(t)+ \frac{4}{1-\beta}\mathcal{X}(t)^{1-\beta},\quad t\geq 0.
\end{equation}
On the other hand, according to \eqref{C11}, the velocity $v_i$ can be expressed as below,
\[v_i(t)=\omega_i+\frac1N \sum_{j=1}^N \Psi(x_j(t) - x_i(t)),\quad t\geq 0.\]
Therefore, we may apply the same arguments above to the estimates of velocity variable and obtain that 
\begin{equation}\label{C22}
\mathcal{V}(t)\leq \mathcal{U}+\frac{4}{1-\beta}\mathcal{X}(t)^{1-\beta},\quad t\geq 0.
\end{equation}
Then, we apply Theorem \ref{L4.1} and \eqref{C21} to obtain the estimate of $\mathcal{X}$ as follows,
\begin{equation}\label{C23}
\mathcal{X}(t)\leq e^{-\psi(2D_0)t}\mathcal{X}(0)+\frac{\mathcal{U}}{\psi(2D_0)}\leq C\left(\mathcal{X}(0)+\mathcal{V}(0)+\mathcal{X}(0)^{1-\beta}\right).
\end{equation}
Finally, we combine \eqref{C21}, \eqref{C22} and \eqref{C23} to obtain the estimate of $\mathcal{V}$ as below,
\[\mathcal{V}(t)\leq C\left(\mathcal{V}(0)+\mathcal{X}(0)^{1-\beta}+\left(\mathcal{X}(0)+\mathcal{V}(0)+\mathcal{X}(0)^{1-\beta}\right)^{1-\beta}\right).\]
\end{proof}
%
%
%
%
%
%
%
%

\section{Measure-valued solutions and mean-field limit for singular CS model}\label{sec:measure}
In this section, we provide the details on the construction of the measure-valued solutions to the kinetic equation \eqref{k_CS}. As mentioned before, we apply the uniform-in-time stability estimates of solutions for the particle system \eqref{p_CS}. We also establish the rigorous mean-field limit estimate from the particle system \eqref{p_CS} to the kinetic equation \eqref{k_CS}.

\begin{theorem}\label{T3.4}
Suppose that initial probability measure $\mu_0\in\mathcal{P}(\mathbb R\times\mathbb R)$ is of compact support for both $x$ and $v$ variables. Moreover, we assume
\begin{align}
\begin{aligned} \label{moments2}
& \iint_{\R \times \R}x  \, \mu_0(dxdv)= \iint_{\R \times \R} v  \, \mu_0(dxdv)=0 \quad \mbox{and} \quad \iint_{\mathbb R\times\mathbb R} \mu_0(dxdv) \leq m_0
\end{aligned}
\end{align} 
for some $m_0 > 0$. Then, the following assertions hold for $1\leq p\leq +\infty$.
\begin{enumerate}
\item[(i)]
There exists a measure-valued solution $\mu_t\in L^{\infty}\left((0,\infty);\mathcal{P}(\mathbb R\times\mathbb R)\right)$ to \eqref{k_CS} with initial data $\mu_0$ such that $\mu_t$ is approximated by empirical measure $\mu_t^N$ associated to \eqref{p_CS} in $p$-Wasserstein distance uniformly in time: \[\varlimsup_{N\rightarrow +\infty}\sup_{t\in[0,+\infty)}W_{p}(\mu_t^N,\mu_t)=0.\]

 \item[(ii)] Moreover, if ${\tilde \mu}_t$ is another measure-valued solution to \eqref{k_CS} with initial measure ${\tilde \mu}_0$ with compact support and zero mean \eqref{moments2}, which is also constructed by empirical measure ${\tilde \mu}_t^N$ in $p$-Wasserstein distance uniformly in time. Then there exists a nonnegative constant $G$ independent of $t$ such that
\[ W_{p}(\mu_t, {\tilde \mu}_t)  \leq G W_{p}(\mu_0, {\tilde \mu}_0), \quad t \in [0, \infty).\]

\item[(iii)] In addition, for every measure-valued solution $\mu_t$ to \eqref{k_CS} constructed from the above mean-field limit process, we can find a positive number $C_0$ independent of $t$ and have the following estimates of the large time behavior:
\[
\lt( \iint_{\R \times \R} |v|^p \,\mu_t(dxdv)\rt)^{\frac{1}{p}} \leq  e^{-\psi(C_0)t} \lt( \iint_{\R \times \R} |v|^p \,\mu_0(dxdv)\rt)^{\frac{1}{p}}.
\] 
\end{enumerate}
\end{theorem}

\begin{remark} Theorem \ref{T3.4} says that the construction of $\mu_t$ is independent of choice of approximation sequence. This implies the solution $\mu_t$ is unique in the sense of mean field limits.
\end{remark}

\begin{proof}[Proof of Theorem \ref{T3.4}]
Since the overall proof of Theorem \ref{T4.1} is almost the same as that of \cite[Corollary 1.1]{HKZ18}, we will provide only sketch of the proof.\\

\noindent $\bullet$~Step A (Extraction of Cauchy approximation for $\mu_0$ in $W_p$): We take a sequence of empirical measures $\tilde{\mu}_0^N=\sum_{i=1}^{N} m_i \delta_{(\tilde x_i , \tilde v_i )}(x,v)$ that approximate $\mu_0$ satisfying
\begin{equation*}
\lim\limits_{N\rightarrow+\infty}W_p(\tilde{\mu}_0^N,\mu_0)=0.
\end{equation*}
The existence of such approximation is guaranteed by \cite{V08}. Moreover, we can make the approximation measure have zero mean. Actually, we denote the mean of position and natural velocity as 
\[
\tilde x_c^N := \sum_{i=1}^{N} m_i \tilde x_i \quad \mbox{and} \quad \tilde v_c^N := \sum_{i=1}^{N} m_i \tilde  v_i,
\]
respectively. Then we consider $\mu^N_0$ as 
\[
\mu^N_0 := \sum_{i=1}^{N} m_i \delta_{(\tilde x_i -\tilde x_c^N, \tilde v_i - \tilde v_c^N)}(x,v) \quad \mbox{with} \quad m_i = \iint_{\mathcal{S}_i} \mu_0 (dxdv).
\]
As $\mu_0$ is a probability measure, we get
\[
\iint_{\R \times \R} \mu_0\,(dxd v) = 1.
\] 
This gives that the total mass of $\mu_0^N$ is $\sum_{i=1}^{N} m_i =1$. Moreover, the expectation of both $x$ and $ v$ with respect to $\mu_0^N$ will be zero, i.e., 
\[
\iint_{\R \times \R} x \mu^N_0 \,(dxd v) = \sum_{i=1}^{N^2} m_i(\tilde x_i - \tilde x_c^N) =0 \quad \mbox{and} \quad \iint_{\R \times \R}  v \mu^N_0 \,(dxd v) = \sum_{i=1}^{N^2} m_i(\tilde  v_i - \tilde  v_c^N) =0.
\]
On the other hand, we find the estimates of $\tilde x_c^N$ as follows:
\begin{align*}
\lim_{N\rightarrow+\infty}\lt| \tilde x_c^N \rt|&= \lim_{N\rightarrow+\infty}\lt|\tilde x_c^N - \iint_{\R \times \R} x \mu_0\,(dxd v)\rt|\cr
&= \lim_{N\rightarrow+\infty}\lt|\sum_{i=1}^{N} m_i \tilde x_i - \iint_{\R \times \R} x \mu_0\,(dxd v)\rt|\\
&= \lim_{N\rightarrow+\infty}\lt|\iint_{\R \times \R} x \tilde{\mu}_0^N\,(dxd v) - \iint_{\R \times \R} x \mu_0\,(dxd v)\rt|\\
&\leq \lim_{N\rightarrow+\infty}W_1(\tilde{\mu}^N_0, \mu_0)\cr
&\leq \lim_{N\rightarrow+\infty}W_p(\tilde{\mu}^N_0, \mu_0)\cr
&=0.
\end{align*}
Similarly, we also have
\[
\lim_{N\rightarrow+\infty}\lt|\tilde  v_c^N\rt| =0.
\]
This yields
\begin{align*}
\lim_{N\rightarrow+\infty}W_p(\mu_0, \mu^N_0)&\leq \lim_{N\rightarrow+\infty}W_p(\mu_0, \tilde{\mu}^N_0)+\lim_{N\rightarrow+\infty}W_p(\tilde{\mu}^N_0, \mu^N_0)\\
&\leq 0+\lim_{N\rightarrow+\infty}\lt( \lt|\tilde x_c^N\rt|^p+\lt|\tilde  v_c^N\rt|^p \rt)^{\frac{1}{p}}=0.
\end{align*}
Hence, we successfully construct an approximation sequence $\mu_0^N$ with zero mean property. Then, for any $\varepsilon>0$, there exists a positive integer $N= N(\varepsilon)$ such that
\[
W_p(\mu^n_0,\mu^m_0)<\varepsilon \quad \mbox{for} \quad n,m> N(\varepsilon).
\]

\noindent $\bullet$~Step B (Approximation of $W_p(\mu_0^n,\mu_0^m)$): Using the argument used in the proof of \cite[Corollary 1.1]{HKZ18}, we can find a natural number $M_{mn}$ such that
\begin{equation}\label{F-2}
\left|W_p^p(\mu_0^n,\mu_0^m)-\frac{1}{M_{mn}}\sum_{k=1}^{M_{mn}}|z_{k0}-\bar{z}_{k0}|^p\right|\le \varepsilon^p,
\end{equation}
where, $z_{k0}:=(x_{k0}, v_{k0})$ and $\bar{z}_{k0}:=(\bar{x}_{k0},\bar{ v}_{k0})$ are support of initial approximated empirical measures $\mu_0^n$ and $\mu_0^m$ respectively.\\

\noindent $\bullet$~Step C (Lifting the information at $s=0$ to $s=t>0$): Now, using \eqref{F-2} and the previous $\ell_p$-stability in particle level in Theorem \ref{L4.1}, we can directly estimate $W_p(\mu_t^n,\mu_t^m)$ as 

\begin{equation}\label{F-3}
W^p_p(\mu_t^n,\mu_t^m)\le 2^{p-1}G^p(W_p^p(\mu_0^n,\mu_0^m)+\varepsilon^p)\leq 2^pG^p\varepsilon^p,
\end{equation}
which implies that the sequence $\mu_t^n$ is Cauchy in $W_p$-metric. Thus, we can find a limit measure $\mu_t$. We next apply similar arguments in \cite{HKZ18} and show that the limit measure $\mu_t$ is the measure-valued solution of the kinetic equation \eqref{KCS1} with initial data $\mu_0$. Moreover, because of the estimate \eqref{F-3}, we can conclude that for any $\varepsilon$, there exists a positive constant $L$, such that
\[\sup_{t\in[0,+\infty)}W_p(\mu^n_t,\mu_t)\leq 4G\varepsilon\quad \mbox{for}\quad n>L.\]
This yields
\begin{equation}\label{F-3-1}
\varlimsup_{N\rightarrow +\infty}\sup_{t\in[0,+\infty)}W_p(\mu_t^N,\mu_t)=0.
\end{equation}
The uniform compact support of $\mu_t$ follows this uniform convergence. As the coefficient $G$ is independent of $p$, we can directly conclude \eqref{F-3-1} holds for $W_\infty$.\\

\noindent $\bullet$~Step D (Uniform stability of kinetic equation): For measures $\mu_0$ and $ \tilde \mu_0$ in ${\mathcal P}(\mathbb R\times\mathbb R)$, let $\mu$ and ${\tilde \mu}$ be measure-valued solutions to \eqref{KCS1}. Then, it follows from \eqref{F-3-1} that for any $\varepsilon \ll 1$, there exists $N_0(\varepsilon) \in \mathbb N$ such that
\[  W_p(\mu, \mu^n) < \frac{\varepsilon}{2}, \qquad  W_p({\tilde \mu}^n, {\tilde \mu}) <  \frac{\varepsilon}{2}, \quad \text{and} \quad n \geq N_0(\varepsilon).  \]
Then, we use the above estimates and \eqref{F-3} to obtain
\begin{align*}
W^p_p(\mu_t, {\tilde \mu}_t) &\leq \Big( W_p(\mu_t, \mu_t^n) + W_p(\mu_t^n, {\tilde \mu}_t^n) + W_p({\tilde \mu}_t^n, {\tilde \mu}_t) \Big)^p \\
&\leq \Big( \varepsilon + W_p(\mu_t^n, {\tilde \mu}_t^n) \Big)^p \\
&\leq 2^{p-1} \Big(\varepsilon^p + W^p_p(\mu_t^n, {\tilde \mu}_t^n) \Big) \\
&\leq 2^{p-1} \Big( 2\varepsilon^p + G^p W_p^p(\mu^n_0, {\mu}^n_0) \Big).
\end{align*}
Letting $n \to \infty$, we have
\[ W^p_p(\mu_t, {\tilde \mu}_t) \leq 2^{p} \varepsilon^p + 2^{p-1} G^p W_p^p(\mu_0,{\tilde \mu}_0).   \]
Since $\varepsilon$ was arbitrary, we have the uniform $W_2$-stability:
\[ W_p(\mu_t, {\tilde\mu}_t) \leq 2^{\frac{p-1}{p}} G W_p(\mu_0, {\tilde \mu}_0), \quad t \geq 0. \]
Also, as the coefficient is uniform with respect to $p$, above stability holds in $W_\infty$.\newline

\noindent $\bullet$~Step E (Large time behavior): For the particle system, the expectation of $|v|^p$ with respect to the measure $\mu_t^N$ is nothing but the $\ell_p$ norm of $v_i$. According to the analysis in \cite{HKZ18,ZZ20}, the $\ell_p$ norm of $v_i$ will decrease to zero exponentially fast with a rate independent of $N$. More precisely, we have
  \begin{equation}\label{D5}
\lt( \iint_{\R \times \R} |v|^p \,\mu_t^N(dxdv)\rt)^{\frac{1}{p}} \leq  e^{-\psi(C_0)t} \lt( \iint_{\R \times \R} |v|^p \,\mu_0^N(dxdv)\rt)^{\frac{1}{p}},
\end{equation}
where $C_0$ is a positive constant independent of $N$. On the other hand, from Step C, we know $\mu_t^N$ will converge to $\mu_t$ in $p$-Wasserstein distance, which immediately implies the weak convergence of $\mu_t^N$ to $\mu_t$ \cite{V08}. Therefore, we let $N$ tends to infinity on both sides of \eqref{D5} and obtain 
\[
\lt( \iint_{\R \times \R} |v|^p \,\mu_t(dxdv)\rt)^{\frac{1}{p}} \leq  e^{-\psi(C_0)t} \lt( \iint_{\R \times \R} |v|^p \,\mu_0(dxdv)\rt)^{\frac{1}{p}}.
\] 
\end{proof}

Next, we provide some remarks on the first-order equation \eqref{D1}. Recall the the continuum equation corresponding to the first-order particle system \eqref{D1}:
\begin{align}\label{KCS1}
\left\{\begin{aligned}
&\partial_tf +\partial_{x} (v[f] f)=0,\quad (x,\omega)\in\mathbb R\times\mathbb R,~~t > 0, \\
& v[f] (x, \omega, t)=\omega+\iint_{\R \times \R} \Psi(x_*-x)f(x_*,\omega_*,t ) \, d\omega_* dx_*.
\end{aligned}
\right.
\end{align}
Note that the probability density function $h = h(\omega)$ for natural velocities appears as a $\omega$-marginal density function of $f$:
\[ 
\int_\R f(x,\omega,t)\,dx = h(\omega).
\] 
In order to establish the global-in-time existence of measure-valued solutions to the kinetic equation \eqref{k_CS}, the uniform-in-time stability estimate of solutions to the particle system \eqref{p_CS} played an essential role. Since we already derived the stability estimate for the first-order particle system, by using almost the same argument as in the previous part, we obtain the uniform-in-time existence of measure-valued solutions to the equation \eqref{main_eq} and the uniform-in-time mean-field limit from \eqref{D1} to \eqref{KCS1}.

\begin{theorem}\label{T4.1}
Suppose that initial probability measure $\sigma_0\in\mathcal{P}(\mathbb R\times\mathbb R)$ is of compact support for both $x$ and $\omega$ variables. Moreover, we assume
\begin{align}
\begin{aligned} \label{moments2}
& \iint_{\R \times \R}x  \, \sigma_0(dxd\omega)= \iint_{\R \times \R} \omega  \, \sigma_0(dxd\omega)=0 \quad \mbox{and} \quad  \iint_{\mathbb R\times\mathbb R} \sigma_0(dxd\omega) \leq m_0.
\end{aligned}
\end{align} 
Then, the following assertions hold for $1\leq p\leq +\infty$.
\begin{enumerate}
\item[(i)]
There exists a measure-valued solution $\sigma_t\in L^{\infty}\left((0,\infty);\mathcal{P}(\mathbb R\times\mathbb R)\right)$ to \eqref{KCS1} with initial data $\sigma_0$ such that $\sigma_t$ is approximated by empirical measure $\sigma_t^N$ in $p$-Wasserstein distance uniformly in time: \[\varlimsup_{N\rightarrow +\infty}\sup_{t\in[0,+\infty)}W_{p}(\sigma_t^N,\sigma_t)=0.\]

 \item[(ii)] Moreover, if ${\tilde \sigma}_t$ is another measure-valued solution to \eqref{KCS1} with initial measure ${\tilde \sigma}_0$ with compact support and zero mean \eqref{moments2}, which is also constructed by empirical measure ${\tilde \sigma}_t^N$ in $p$-Wasserstein distance uniformly in time. Then there exists nonnegative constant $G$ independent of $t$ such that
\[ W_{p}(\sigma_t, {\tilde \sigma}_t)  \leq G W_{p}(\sigma_0, {\tilde \sigma}_0), \quad t \in [0, \infty).\]
\end{enumerate}
\end{theorem}
Note that the stability result in Theorem \ref{T4.1} immediately implies the existence of limit measure $\sigma_{\infty}$ when $t$ tends to infinity. More precisely, we have the following corollary. 
\begin{corollary}\label{C4.1}
Suppose that initial probability measure $\sigma_0\in\mathcal{P}(\mathbb R\times\mathbb R)$ is of compact support for both $x$ and $\omega$ variable. Let us denote by $\Sigma_x(\sigma_t)$ and $\Sigma_\omega(\sigma_t)$ the $x$- and $\omega$-projections of support of $\sigma_t$. Moreover, we assume
\begin{align}
\begin{aligned} \label{moments2}
& \iint_{\R \times \R} x  \, \sigma_0(dxd\omega)= \iint_{\R \times\R}\omega  \, \sigma_0(dxd\omega)=0.
\end{aligned}
\end{align} 
Then, the solution $\sigma_t$ is unique in the sense of mean field limits. Moreover, for the initial measure $\sigma_0$, there exists a limit measure $\sigma_\infty$, which is unique in the sense of mean field limits, such that 
\[
W_p(\sigma_t,\sigma_\infty)\leq C_1 e^{-\psi(C_0)t}
\]
for $t \geq 0$, where the positive constants $C_0$ and $C_1$ are given by
\[
C_0 := \max\{diam(\Sigma_x(\sigma_0)), \Psi^{-1}(diam(\Sigma_\omega(\sigma_0))) \}
\] 
and
\[
C_1 :=2\max_{(x,\omega) \in supp(\sigma_0)} \lt|\omega+ \iint_{\R \times \R} \Psi(x_* - x)\sigma_0(dx_* d\omega_*)\rt|,
\]
respectively.
\end{corollary}
\begin{remark}
We construct a measure-valued solution $\mu_t$ to \eqref{k_CS} in Theorem \ref{T3.4}, and on the other hand, we can also construct a measure-valued solution $\Gamma(\cdot, \cdot,t) \# \sigma_t$ to \eqref{k_CS} as discussion in the introduction, where $\sigma_t$ is the measure-valued solution to \eqref{main_eq}. So far, given $\mu_0=\Gamma(\cdot, \cdot,0) \# \sigma_0$, we cannot say $\mu_t=\Gamma(\cdot, \cdot,t) \# \sigma_t$ since it is not clear that $\Gamma(\cdot, \cdot,t) \# \sigma_t$ can be approximated by the empirical measure associated to the particle system \eqref{p_CS}. We only have
\[
\varlimsup_{N\rightarrow +\infty}\sup_{t\in[0,+\infty)} W_p(\Gamma^N_t \# \sigma^N_t, \mu_t) =0,
\]
where $\sigma^N_t$ is given as in \eqref{eq_mu}. In next section, we will see these two solutions coincide if we add some regularity conditions. 
\end{remark}

%
%
%
\section{Contractivity in a modified $p$-Wasserstein distance}\label{sec:contr}
\setcounter{equation}{0}
\vspace{0.3cm}

In this section, we provide the contractivity of solutions to the first-order reformulation \eqref{main_eq} by employing the method of optimal mass transport. 

\subsection{Reformulation of \eqref{main_eq}: Pseudo-inverse}

Define the cumulative distribution $\tilde f_t$ of the probability measure $f_t \in \mathcal{P}(\R \times \R)$ satisfying \eqref{main_eq} in the sense of Definition \ref{D4.1} by:
\[
\tilde f_t(x,\omega) := \int_{-\infty}^x f_t(dy,\omega).
\]
Let $h$ be the second marginal of $f_t$, i.e.,  $h_t= \pi^2 \# f_t$, then it follows from \eqref{main_eq} that 
\[
\pa_t h_t =  \pa_t\intr f_t(dx,\omega) = 0,
\]
thus we get
\[
(\pi^2 \# f_t)(d\omega) = (\pi^2 \# f_0)(d\omega) =: h(\omega)
\]
for $t\geq 0$. We also define the pseudo-inverse $\chi_f$ of $\tilde f$ as a function of $x$:
\[
\chi_f(\eta,\omega,t) := \inf \{x : \tilde f_t(x,\omega) > \eta\}, \quad \eta \in [0,h(\omega)].
\]
Note that
\[
\tilde f_t(\chi_f(\eta,\omega,t),\omega) = \eta.
\]
Let us try to find an equation for $\chi_f$. By definition, we first easily get
\[
\intrr \varphi(x,\omega,t)\,f_t(dxd\omega) = \int_\R\int_0^{h(\omega)} \varphi(\chi_f(\eta,\omega,t),\omega,t)\,d\eta d\omega.
\]
Using this together with the weak formulation in Definition \ref{D4.1} (iii), we have
\begin{align}\label{eq_chi}
\begin{aligned}
&\int_\R\int_0^{h(\omega)} \varphi(\chi_f(\eta,\omega,t),\omega,t)\,d\eta d\omega - \int_\R\int_0^{h(\omega)} \varphi(\chi_f(\eta,\omega,0),\omega,0)\,d\eta d\omega\cr
&\quad = \intrr \varphi(x,\omega,t)\,f_t(dxd\omega) - \intrr \varphi(x,\omega,0)\,f_0(dxd\omega)\cr
&\quad = \int_0^t \intrr \lt(\pa_s \varphi + (\omega - \Psi \star\rho_s)\pa_x \varphi \rt) f_s(dxd\omega)\,ds\cr
&\quad = \int_0^t \int_\R \int_0^{h(\omega)} ((\pa_s \varphi)(\chi_f(\eta,\omega,s),\omega,s) \cr
&\hspace{3cm} + (\omega - (\Psi \star \rho_s)(\chi_f(\eta,\omega,s)))(\pa_x \varphi)(\chi_f(\eta,\omega,s),\omega,s) )\,d\eta d\omega ds.
\end{aligned}
\end{align}
On the other hand, 
\[
\pa_t  \varphi(\chi_f(\eta,\omega,t),\omega,t) = (\pa_t \varphi)(\chi_f(\eta,\omega,t),\omega,t) + (\pa_x \varphi)(\chi_f(\eta,\omega,t),\omega,t)\pa_t \chi_f(\eta,\omega,t),
\]
and this implies
$$\begin{aligned}
&\int_0^t \int_\R \int_0^{h(\omega)} (\pa_s \varphi)(\chi_f(\eta,\omega,s),\omega,s)\,d\eta d\omega ds\cr
&\quad = \int_\R\int_0^{h(\omega)} \varphi(\chi_f(\eta,\omega,t),\omega,t)\,d\eta d\omega - \int_\R\int_0^{h(\omega)} \varphi(\chi_f(\eta,\omega,0),\omega,0)\,d\eta d\omega \cr
&\qquad - \int_0^t \int_\R \int_0^{h(\omega)}(\pa_x \varphi)(\chi_f(\eta,w,s),w,s)\pa_s \chi_f(\eta,w,s)\,d\eta dw ds.
\end{aligned}$$
This and \eqref{eq_chi} yield
\[
\pa_t \chi_f(\eta,\omega,t) = w - (\Psi \star \rho_t)(\chi_f(\eta,\omega,t)).
\]
Since
$$\begin{aligned}
(\Psi \star \rho_t)(\chi_f(x,\omega,t)) &= \int_\R \Psi(\chi_f(\eta,\omega,t)-y) \rho_t(dy) \cr
&= \intrr \Psi(\chi_f(\eta,\omega,t)-y)f_t(dyd\omega_*)\cr
&= \intr  \int_0^{h(\omega_*)} \Psi(\chi_f(\eta,\omega,t)-\chi_f(\eta_*,\omega_*,t))\,d\eta_* d\omega_*,
\end{aligned}$$
we have the following integro-differential equation:
\bq\label{eq_chi2}
\pa_t \chi_f(\eta,\omega,t) = \omega + \intr  \int_0^{h(\omega_*)}\Psi(\chi_f(\eta_*,\omega_*,t) - \chi_f(\eta,\omega,t))\,d\eta_* d\omega_*.
\eq

We next introduce a notion of absolutely continuous solutions to the equation \eqref{main_eq} below. Note that it is straightforward to find that $\chi_f$ satisfies the integro-differential equation \eqref{eq_chi2} with the absolutely continuous solution $f_t$ of \eqref{main_eq}, see for instance \cite{BCDP15}. 

\begin{definition}\label{def_aconti} A curve $f_t \in AC([0,+\infty); \mathcal{P}_p(\R \times \R))$ is called absolutely continuous solution if 
\begin{itemize}
\item[(i)]
\[
\pa_t f + \pa_x ((\omega - \Psi \star \rho)f) = 0 \quad \mbox{in} \quad \mathcal{D}'(\R \times \R \times [0,+\infty))
\]
and
\item[(ii)] 
\[
\int_0^T \|\omega - (\Psi \star \rho_t)(x)\|_{L^2(f_t)}\,dt < \infty
\]
for every $T>0$.
\end{itemize}
\end{definition}

Then we can combine Theorem \ref{T4.1} and Lemma \ref{lem_energy} (iii) to have the following existence and uniqueness of absolutely continuous solutions of \eqref{main_eq} in the sense of Definition \eqref{def_aconti}, and subsequently this implies that \eqref{eq_chi2} is well-defined  \cite{BCDP15, CCT19, CFFLS11}. Notice that in Theorem \ref{T4.1} we assumed that the initial data $\mu_0$ has compact support in $x$ and $\omega$ variables, and the uniform size of support can be obtained. Thus any $p$-th moments of the measure-valued solution is bounded uniformly in time. Under these assumptions, we can also use the observation in Remark \ref{rmk_free} for the condition (ii) in Definition \ref{def_aconti}, instead of Lemma \ref{lem_energy} (iii).

\subsection{Contractivity estimate}

Note that the $p$-Wasserstein distance $W_p(f, g)$ between two measures $f$ and $g$ is equivalent to the $L^p$-distance between the corresponding pseudo-inverse functions $\chi_f$ and $\chi_g$, respectively. In this respect, we set
\[
W_p(f,g)(\omega,t) := \|\chi_f(\cdot,\omega,t) - \chi_g(\cdot,\omega,t)\|_{L^p(0,h(\omega))}
\]
for $1 \leq p \leq \infty$. We then define a modified transport distance 
\[
\W_p(f,g)(t) := \|W_p(f,g)(\cdot,t)\|_{L^p}
\]
for $1 \leq p < \infty$.  We assume that $h$ is compactly supported, then we can also define
\[
\W_\infty(f,g)(t) := \lim_{p \to \infty}\|W_p(f,g)(\cdot,t)\|_{L^p}.
\]
By definition, for fixed $\omega \in \R$ we obtain
\[
\chi_f(0,\omega,t) \leq \chi_f(\eta,\omega,t) \leq \chi_f(h(\omega),\omega,t)
\]
for all $\eta \in [0,h(\omega)]$ and $t \geq 0$. This entails 
\[
|\chi_f(\eta,\omega,t) - \chi_f(\eta_*,\omega_*,t)| \leq \sup_{\omega \in supp(h)} \chi_f(h(\omega),\omega,t) - \inf_{\omega \in supp(h)} \chi_f(0,\omega,t).
\]
Moreover, we also easily find that the right hand side of the above inequality equals the diameter of the support for $f$ in $x$. Thus we have
\[
|\chi_f(\eta,\omega,t) - \chi_f(\eta_*,\omega_*,t)| \leq 2D_0.
\]
For notational simplicity, we set
\[
\chi_f := \chi_f(\eta,\omega,t) \quad \mbox{and} \quad \chi_{f_*} := \chi_f(\eta_*,\omega_*,t).
\]
Then our main result of this section is as follows.
\begin{theorem}\label{thm_cont} Let $f_t$ and $g_t$ be solutions of \eqref{main_eq} in the sense of Definition \ref{def_aconti} and have the compact support in $x$ and $\omega$. Then we have
\[
\W_p(f_t,g_t) \leq \W_p(f_0,g_0) e^{-\psi(2D_0)t}
\]
for $t \geq 0$ and $1 \leq p \leq \infty$.
\end{theorem}
\begin{proof}
A straightforward computation gives
\begin{align}\label{chi_p}
\begin{aligned}
&\frac{d}{dt}\|\chi_f - \chi_g\|_{L^p}^p \cr
&\quad = p\intr \int_0^{h(\omega)} |\chi_f - \chi_g|^{p-2} (\chi_f - \chi_g) \pa_t (\chi_f - \chi_g)\,d\eta d\omega\cr
&\quad = p\intr \intr \int_0^{h(\omega_*)}\int_0^{h(\omega)}  |\chi_f - \chi_g|^{p-2} (\chi_f - \chi_g) \lt( \Psi(\chi_{f_*} - \chi_f) - \Psi(\chi_{g_*} - \chi_g)\rt) d\eta d\eta_* d\omega  d\omega_*\cr
&\quad = \frac p2\intr \intr \int_0^{h(\omega_*)}\int_0^{h(\omega)} \lt(|\chi_f - \chi_g|^{p-2} (\chi_f - \chi_g) - |\chi_{f_*} - \chi_{g_*}|^{p-2} (\chi_{f_*} - \chi_{g_*})\rt)\cr
&\hspace{5cm} \times\lt( \Psi(\chi_{f_*} - \chi_f) - \Psi(\chi_{g_*} - \chi_g)\rt)\,d\eta d\eta_* d\omega d\omega_*.
\end{aligned}
\end{align}
Here we interchanged between $(\eta,\omega)$ and $(\eta_*, \omega_*)$ and used the fact that $\Psi$ is odd. On the other hand, by the mean-value theorem, we get
\[
\Psi(\chi_{f_*} - \chi_f) - \Psi(\chi_{g_*} - \chi_g) = \psi(c) \lt(\chi_{f_*} - \chi_{g_*} - (\chi_f  - \chi_g) \rt),
\]
where $c$ is some point between $\chi_{f_*} - \chi_f$ and $\chi_{g_*} - \chi_g$. Moreover, by supporting hyperplane theorem, we find
\[
(|x|^{p-2} x - |y|^{p-2}y)\cdot (x-y) \geq 0
\]
for $p \geq 1$ and $x,y \in \R^d$.
This asserts that the right hand side of \eqref{chi_p} is bounded from above by
$$\begin{aligned}
& \frac {p \psi(2D_0)}2\intr \intr \int_0^{h(\omega_*)}\int_0^{h(\omega)} \lt(|\chi_f - \chi_g|^{p-2} (\chi_f - \chi_g) - |\chi_{f_*} - \chi_{g_*}|^{p-2} (\chi_{f_*} - \chi_{g_*})\rt)\cr
&\hspace{3cm} \times\lt(\chi_{f_*} - \chi_{g_*} - (\chi_f  - \chi_g) \rt) \,d\eta d\eta_* d\omega d\omega_*\cr
&\quad = -p\psi(2D_0)\intr \int_0^{h(\omega)}|\chi_f - \chi_g|^p \,d\eta d\omega \cr
&\qquad +  \frac {p \psi(2D_0)}2\intr \int_0^{h(\omega)} |\chi_f - \chi_g|^{p-2} (\chi_f - \chi_g) \,d\eta d\omega \intr \int_0^{h(\omega_*)} (\chi_{f_*} - \chi_{g_*})\,d\eta_* d\omega_*\cr
&\qquad +  \frac {p \psi(2D_0)}2\intr \int_0^{h(\omega_*)} |\chi_{f_*} - \chi_{g_*}|^{p-2} (\chi_{f_*} - \chi_{g_*}) \,d\eta_* d\omega_* \intr \int_0^{h(\omega)} (\chi_f - \chi_g) \,d\eta d\omega\cr
&\quad = -p\psi(2D_0)\intr \int_0^{h(\omega)}|\chi_f - \chi_g|^p \,d\eta d\omega,
\end{aligned}$$
where we used
\[
\intr \int_0^{h(\omega)} \chi_f \,d\eta d\omega =  \intr \int_0^{h(\omega)} \chi_g \,d\eta d\omega.
\]
Thus we have
\[
\frac{d}{dt}\|\chi_f - \chi_g\|_{L^p}^p  \leq -p\psi(2D_0)\|\chi_f - \chi_g\|_{L^p}^p,
\]
and this subsequently implies
\[
\frac{d}{dt}\|\chi_f - \chi_g\|_{L^p}  \leq -\psi(2D_0)\|\chi_f - \chi_g\|_{L^p}.
\]
Applying Gr\"onwall's lemma concludes our desired result.
\end{proof}

\begin{remark} Note the uniqueness of absolutely continuous solutions to \eqref{main_eq} is a simple consequence of Theorem \ref{thm_cont}. Moreover, with the change of variable $\Gamma$ in Section \ref{ssec:2.5}, we obtain the uniqueness of the absolutely continuous solutions to \eqref{k_CS}.
\end{remark}

\section{Conclusion}\label{sec:6}
\setcounter{equation}{0}
\vspace{0.3cm}
In this paper, we studied the one dimensional CS model with a singular communication weight. At the particle level, we obtained uniform-in-time stability for system \eqref{p_CS} and \eqref{D1}. Then, we applied the uniform-in-time stability to yield the uniform-in-time mean-field limit from the particle systems  \eqref{p_CS} and \eqref{D1} to the kinetic equation \eqref{k_CS} and \eqref{main_eq}, respectively. Moreover, the large time behavior of the measure-valued solutions can be implied as a directly consequence of the uniform-in-time mean-field limit. Finally, we added absolutely continuity property to the initial data, and applied the pseudo inverse method to obtain the contractivity and uniqueness of the measure-valued solutions. So far, this method cannot be applied to multi-dimensional CS model, since there is no equivalent first-order equation in multi-dimensional case. We will study the stability and the mean-field limit for multi-dimensional singular CS model in future work. 
%
%
%
%
\section*{Acknowledgments}
The work of Y.-P. Choi is supported by NRF grant (No. 2017R1C1B2012918) and Yonsei University Research Fund of 2019-22-021. The work of X. Zhang is supported by the National Natural Science Foundation of China (Grant No. 11801194), the National Natural Science Foundation of China (Grant No. 11971188) and the HUST Research Fund.

%
%
%
%

\end{document}